\documentclass[12pt]{article}
\usepackage[latin1]{inputenc}
\usepackage{vector}
\usepackage{amsmath}
\usepackage{enumerate}
\usepackage{amsfonts}
\usepackage{amssymb}
\usepackage{hyperref}

\newcommand{\qed}{\hfill $\Box$ \vspace{\baselineskip}}

\numberwithin{equation}{section}

\newtheorem{teo}{Theorem}[section]

\newtheorem{df}[teo]{Definition}

\newtheorem{cor}[teo]{Corollary}

\newtheorem{lem}[teo]{Lemma}

\newtheorem{conj}[teo]{Conjecture}

\def\no{\noindent}

\def\sbs{\subset}
\def\es{\emptyset}

\def\td{\tilde}
\def\b{\beta}
\def\a{\alpha}
\def\d{\delta}
\def\ba{\begin{array}}
\def\ea{\end{array}}
\def\nn{\nonumber}
\def\rd{{\rm d}}

\def\t{\theta}

\def\Ra{\Rightarrow}

\def\P{\mathsf{P}}
\def\E{\mathsf{E}}

\def\R{\mathbb{R}}

\def\n{\{0,1\}^N}
\def\nn{[N]}

\begin{document}

\begin{center}
{\huge AI's solution of Conjecture 9.1.}
\end{center}

\medskip

\begin{center}Jinyoung Park\footnote{New York University, US, jinyoungpark@nyu.edu} and
 Michel Talagrand\footnote {Paris, France, michel.talagrand@gmail.com}
\end{center}

\vspace{1cm}

\begin{abstract} In an earlier paper, {M. Talagrand} made a number of daring conjectures. An  AI model provided a very beautiful proof of one of them  which  {we} explain here. 
\end{abstract}

\section{Introduction}

Let us denote by $\nn$ the set of subsets of $\{{1, 2},\ldots, N\}$, which we identify with $\n$. For $0<p<1$ we consider the probability $\mu_p=((1-p)\d_0+p\d_1)^{\otimes N}$ on $\nn=\n$. A main idea of \cite{small} is  {that} if a set $A\sbs \nn$ is not too small, say $\mu_p(A)\geq 1/2$, the set $A^{(2)}$ of elements  {of} $\nn$ which {\bf cannot} be covered by two elements of $A$ should be small.

Let  us recall some notions introduced in  \cite{small}. For an element $I\in \nn$ we denote $H_I=\{J\in \nn; I\sbs J\}$. Thus $\mu_p(H_I)=p^{|I|}$, where $|I|$ denotes the cardinal of $I$. 
\begin{df} A subset $B$ of $\n$ is called $p$-small if $B\sbs \cup_\ell H_{I_\ell}$ where $\sum_\ell p^{|I_\ell|}\leq 1/2.$
\end{df}

A main conjecture of \cite{small} is as follows.
\begin{conj}\lab{conj1}
There exists $\a>0$ such that $A^{(2)}$ is $(\a p)$-small whenever $\mu_p(A)\geq 3/4.$
\end{conj}
This conjecture is still open, although one must say that combining the results of \cite{li} (or the later results explained  {in} this note) and those  {of} \cite{p}  {shows that} the remaining gap is only of order $\log\log(100/p)$.
 A weaker notion of smallness was introduced in \cite{small}.
 \begin{df}Given $\rho>0$ a probability $\nu$ on $\nn$ is $\rho$-spread if $\nu(H_I)\leq \rho^{|I|}$ for each $I\in \nn$.  
 \end{df}
 \begin{df}A subset $B$ of $\nn$ is weakly $\rho$-small if it does not support a $\rho$-spread probability measure. \end{df}
 
 It should be obvious that a $p$-small set is weakly $p$-small. The following daring conjecture of \cite{small} is still open. 
 \begin{conj} There exists  {$\a>0$} such that a  {weakly $p$-small} set is $(\a p)$-small. 
 \end{conj} 
 This conjecture is still open, although it is shown  {in} \cite{p} that the remaining gap is only of order $\log\log(100/p)$. 

  Chen Li \cite{li} recently guided an AI agent to obtain a proof of the weaker version of Conjecture \ref{conj1} where $(\a p)$-small is replaced by  {weakly $p$-small}, so that  \cite{p} show that only a small $\log\log$ gap remains to be closed to  prove Conjecture \ref{conj1}.
 
 \section{Projections} 
 If a subset $A$ of $\nn$ satisfies $\mu_{1/2}(A)> {1/2}$ there is $J\in A$ such that its complement $J^c$ is also in $A$ and thus $A^{(2)}=\es$. Consequently if for $J\in \nn$ the projection of $A$ on $\{0,1\}^J$ has a measure $>1/2$ for the uniform measure, then $J\not \in A^{(2)}$. More generally we define
 \be\lab{9}\t_{ {\a},J}(A)=\mu_\a(\{ L\in \nn; L\cap J\in A\} {)},\ee
 so that $\t_{1/2,J}(A)>1/2\Ra J\not \in A^{(2)}$. Hence the idea that to show that $A^{(2)}$ is small it suffices to show that the function $\t_{1/2,J}(A)$ of $J$ is not too small. 
 
 Throughout the following, we assume that the set $A$ is  {nonempty and} hereditary, that is $I\sbs J\in A\Ra I\in A$.  The following is a strong version of Conjecture 9.1 of \cite{small}. It was obtained by an  {A.I agent\footnote{GPT-6 Astra}} which had been fed the paper \cite{li} before. 
 \begin{teo}\lab{th1} Consider  {$0<\a<1$} and $\rho=p(1-\a)/((1-p)\a)$. Then if the measure $\nu$ is $\rho$-spread, 
 \be\lab{1}\log \mu_p(A)\leq \int \log \t_{\a, J}(A) \rd \nu (J).\ee
 \end{teo} 
 \begin{cor}[Essentially \cite{li}] If $\mu_p(A)>1/2$ then $A^{(2)}$ is  {weakly  $p/(1-p)$-small}, and hence   {weakly $p$-small}.
 \end{cor}
 {\bf Proof.} Note that for $\a=1/2$ {,} $\rho= p/(1-p)$. If $\nu$ is a $p/(1-p)$-spread measure, then \rref{1} shows that there is $J$ in the support of $\nu$ such that $\t_{1/2, J}(A)>1/2$ and as we have seen we have $J\not \in A^{(2)}$ so that $A^{(2)}$  {does not support}
 a $p/(1-p)  $-spread measure. \qed
 
Let us also observe that the  {concavity} of the $\log$ implies the weaker inequality
\be\lab{10}\mu_p(A)\leq   \int  \t_{\a, J}(A) \rd \nu (J).\ee

A  few days ago,   X. Fang and T. Wang \cite{fw} {guided}
an AI agent which, after being fed Li's paper \cite{li}, produced a direct proof of  \rref{10}.

 \section{Proofs}
 We consider the function $\chi$ of a real variable $x$ given by 
 $$\chi(x)=\frac{x-p}{\sqrt{p( {1}-p)}}.$$
 Given $z=(z_i)_{i\leq N}\in \R^N$ and $I\in \nn$ we define $\chi_I(z)=1$ if $I=\es$ and $\chi_I(z)=\prod_{i\in I}\chi(z_i)$ otherwise. The functions $\chi_I$ form an orthonormal basis of $L_2(\n, \mu_p)$. We consider the  {expansion} 
 $${\sf 1}_A=\sum_I a_I \chi_I, $$
 so that 
 $$a_\es=\mu_p(A)  ; \quad \sum_I a_I^2=\mu_p(A).$$
We consider the function 
$$F(z)=\sum_I a_I\chi_I(z)$$
on $\R^N$. The following should be obvious. 
\begin{lem}Consider independent r.v.s $(Z_i)_{i\leq N}$ and assume that for each $i$ we have $\E Z_i=p$. Then
\be\lab{2} \E F(Z)=a_\es=\mu_p(A)\ee
\be\lab{3} \E F(Z)^2=\sum_I a_I^2\prod_{i\in I}\E \chi(Z_i)^2.\ee
\end{lem}
Let us now fix a set $J\in \nn$. For $i\not \in J$ we set $Z_i=p$. For $i\in J$, $Z_i$ takes the value 1 with probability $\a$ and the value $(p-\a)/(1-\a)$ with probability $1-\a$. Thus $\E Z_i=p$ and a straightforward computation yields
$$\E \chi(Z_i)^2=\b:=\frac{\a(1-p)}{(1-\a)p}  {\quad (i \in J)}.$$
Writing $Z_J=(Z_i)_{i\leq N}$ we then get from \rref{2} and \rref{3} that 
\be\lab{4} \E F(Z_J)=\mu_p(A),\ee
\be \lab{5} \E F(Z_J)^2=\sum_{I\sbs J} {a_I^2} \b^{|I|}.\ee
We then use the Cauchy-Schwarz inequality in the form 
\be  {\bigl(\E F(Z_J)\bigr)^2}\leq \P ( {F(Z_J)}\not=0) \E F(Z_J)^2.\ee
 The extraordinarily clever part of the proof is the following observation:
 \begin{lem}\lab{l32}We have $\P( {F(Z_J)}\not =0)\leq \t_{\a, J}(A).$
 \end{lem}
 {\bf Proof of  {Theorem~\ref{th1}}.} We have obtained that 
 $$\mu_p(A)^2\leq \t_{\a, J}(A) \sum_{I\sbs J}a_I^2\b^{|I|},$$
 so that taking logarithms
 \be\lab{6}2\log \mu_p(A)\leq \log \t_{\a, J}(A)+\log \sum_{I \sbs J}a_I^2\b^{|I|}.\ee
 
 Consider then a $\rho$-spread measure, with $\rho\b\leq 1$. Thus 
 $$\int \Bigl( \sum_{I\sbs J}a_I^2 \b^{|I|}\Bigr)\rd \nu(J)\leq \sum_I a_I^2(\b \rho)^{|I|}\leq \sum_I a_I^2=\mu_p(A).$$
  Thus, integrating \rref{6} with respect to $\rd \nu$ and using the concavity of the $\log$ yields 
  $$2\log \mu_p(A)\leq \int \log \t_{\a,J}(A) \rd \nu(J)+\log \mu_p(A),$$
  which is the required inequality. \qed 
  
 \no {\bf Proof of Lemma \ref{l32}.} Consider the random set $I=\{i\in J; Z_i=1\}$. We will prove that 
 \be\lab{7} I\not \in A \Ra F(Z_J)=0.\ee
 Indeed, when this holds it follows  that $F(Z_J)\not=0\Ra I\in A$, and comparison with \rref{9} shows that the probability of this  {latter} event is exactly $\t_{\a,J}(A)$. 
 
  A function on $\R^N$ of the type $\sum_I b_I\chi_I(z)$ is entirely determined by the values it takes for $z\in \{0,1\}^N$ because the $b_I$ are the Fourier coefficients of this function seen as an element of  $L^2(\{0,1\}^N, \mu_p)$. In particular if the function is zero for each $z \in \{0,1\}^N$ it is identically zero. Let us then consider the function $F(z)$ when we fix $z_i=1$ for $i\in I$ as a  {function $F'(z')$} of the other variables $z'\in \R^{I^c}$. It is of the type  {$\sum_{K \subset I^c} b_K \chi_K(z')$}. When $I\not \in A$ we want to prove that this function is identically zero, and for this it suffices to prove that it is zero when the components of $z'$ take the values 0 or 1. But  then $z\in \n$ and $F(z)={\sf 1}_A(z)=0$ because $A$ is hereditary so that $z\not \in A$ because $I\not \in A$ and $I \sbs z$.\qed

  \section*{Acknowledgment} The proofs presented in this document were generated entirely by the AI model GPT-6 Astra, and none of the authors claim any credit for them. This write-up is written solely as a service to the research community and is not intended for journal publication.

\end{document}